\documentclass[reqno]{amsart}
\usepackage{amssymb}
\usepackage{hyperref}
\usepackage{booktabs}
 \usepackage{multirow}

\newtheorem{theorem}{Theorem}[section]
\newtheorem{proposition}[theorem]{Proposition}

\theoremstyle{remark}
\newtheorem{remark}{Remark}[section]

\numberwithin{equation}{section}

\begin{document}

\title[Discrete Fourier Transform and Schur polynomials]
{Discrete Fourier transform associated with generalized Schur polynomials}

\author{J.F.  van Diejen}

\address{
Instituto de Matem\'atica y F\'{\i}sica, Universidad de Talca,
Casilla 747, Talca, Chile}

\email{diejen@inst-mat.utalca.cl}

\author{E. Emsiz}

\address{
Facultad de Matem\'aticas, Pontificia Universidad Cat\'olica de Chile,
Casilla 306, Correo 22, Santiago, Chile}
\email{eemsiz@mat.uc.cl}


\thanks{This work was supported in part by the {\em Fondo Nacional de Desarrollo
Cient\'{\i}fico y Tecnol\'ogico (FONDECYT)} Grants   \# 1141114 and \# 1170179.}

\date{November 2017}

\begin{abstract}
We prove the Plancherel formula for a four-parameter family of discrete Fourier transforms and their multivariate generalizations stemming from
corresponding generalized Schur polynomials. For special choices of the parameters, this recovers 
 the sixteen classic discrete sine- and cosine transforms DST-1,$\ldots$,DST-8 and DCT-1,$\ldots$,DCT-8, as well as recently  studied (anti-)symmetric multivariate generalizations thereof.
\end{abstract}

\maketitle



\section{Introduction}\label{sec1}
Apart from its profound theoretical significance, the discrete Fourier transform provides an effective computational tool for performing numerical harmonic analysis in applied contexts  \cite{aus-gru:fourier,ter:fourier,won:discrete,goo:discrete}. In its most conventional form,
the kernel of the discrete Fourier transform (DFT) with period $m$ in $n$ variables is given by the eigenfunctions of a discrete Laplacian on the periodic lattice $\mathbb{Z}^n/ m\mathbb{Z}^n$. By restricting to the space of (permutation) symmetric or anti-symmetric functions on this lattice, one is led to finite-dimensional discrete orthogonality relations for respectively the monomial symmetric polynomials and the Schur polynomials, cf. e.g.  \cite[Sec. 5]{kli-pat:anti-a} or \cite[Sec. 8.4]{die-ems:discrete13}.  In mathematical physics, such finite-dimensional discrete orthogonality structures arise naturally in the study of certain quantum-integrable $n$-particle systems on the one-dimensional periodic lattice $\mathbb{Z}/m\mathbb{Z}$, cf. e.g. \cite[Sec. 5.2]{die-vin:quantum}, \cite[Sec. 5.2]{die:diagonalization}  and   \cite[Sec. 10]{kor-stro:slnk}.

In this note we consider  a discrete Fourier transform on a finite aperiodic integral grid consisting of the nodes $ \{ 0,\ldots ,m\}$ (with $m>0$). It stems from the eigenfunctions of a perturbation of the discrete Laplacian characterized by  homogeneous two-parameter boundary conditions at both ends of the grid. For specific values of the boundary parameters,  the  celebrated sixteen discrete sine- and cosine transforms DST-1,$\ldots$,DST-8 and DCT-1,$\ldots$,DCT-8 
\cite{wan-hun:discrete,str:discrete,bri-yip-rao:discrete} are recovered. A distinctive characteristic of our discrete Fourier transform is that the spectrum of the underlying Laplacian depends algebraically on the values of the boundary parameters. The DST-$k$ and DCT-$k$ ($k=1,\ldots ,8$) correspond from this perspective to well-known boundary conditions of Dirichlet- and Neumann type, in which cases the gaps in  the spectral parameter become equidistant. For general boundary parameters the discrete Fourier kernel turns out to be given by special instances of the Bernstein-Szeg\"o orthogonal polynomials
\cite{sze:orthogonal,gau-not:gauss,gri:oscillatory}. By building the associated generalized Schur orthogonal
polynomials \cite{mac:schur,nak-nou-shi-yam:tableau,ser-ves:jacobi} (cf. also Refs. \cite{las:jacobi,las:laguerre,las:hermite,ols:laguerre,cor-wil:macdonald} for examples of analogous generalized Schur polynomials associated with the classical families of hypergeometric orthogonal polynomials), we deduce the Plancherel formula for a multivariate discrete Fourier transform that specializes in the cases of Dirichlet- and Neumann boundary conditions to special instances of intensely studied
(anti-)symmetric extensions of the discrete (co)sine transforms 
\cite{kli-pat:anti-a,kli-pat:anti-b,li-xu:discrete,moo-pat:cubature,moo-mot-pat:gaussian,czy-hri:generalized,hri-mot:discrete}.
From the point of view of mathematical physics, our results settle the Plancherel problem for a phase model of strongly correlated bosons on the lattice
\cite{bog:boxed,die:diagonalization,kor-stro:slnk}  endowed with open-end boundary interactions \cite{die-ems-zur:completeness}.

The presentation breaks up in two parts: Section \ref{sec2} and Section \ref{sec3}, treating the univariate- and the multivariate setup, respectively.

\section{Perturbation of the discrete (co)sine transforms at the boundary}\label{sec2}
By diagonalizing a one-dimensional discrete Laplacian with homogeneous two-parameter boundary conditions at the lattice-ends, we arrive
at a four-parameter family of discrete Fourier transforms. This family unifies the sixteen standard discrete sine-- and cosine transforms DST-1,\ldots,DST-8  and DCT-1,\ldots,DCT-8 \cite{str:discrete,bri-yip-rao:discrete}, which are recovered
for special values of the boundary parameters.

\subsection{Discrete Laplacian}
For $m\in\mathbb{N}$, let  $\Lambda^{(m)}:=\{ 0,1,\ldots ,m\}$. We consider the following action of the  $(m+1)\times (m+1)$ tridiagonal matrix
\begin{equation}\label{L}
\text{L}^{(m)}=
\begin{bmatrix}
b_- &  1-a_-  & 0&  \cdots & 0\\
1  &  0 & 1& & \vdots \\
0 & \ddots & \ddots & \ddots &0\\
\vdots  &   &  1       &0&1  \\
0 & \cdots & 0& 1-a_+ & b_+
\end{bmatrix}
\end{equation}
on the ($m+1$)-dimensional space $C(\Lambda^{(m)})$ of functions
$f:\Lambda^{(m)}\to\mathbb{C}$:
\begin{equation}\label{L-action}
(\text{L}^{(m)}f) (l) =  
\begin{cases}   \bigl(1-a_-\bigr) f(l+1)+ 
  b_-  f(l)  &\text{if}\ l=0, \\
 f(l+1)+f(l-1) 
&\text{if}\ 0<l<m, \\
 \bigl( 1-a_+\bigr)f(l-1) 
 + b_+ f(l) &\text{if}\ l=m.
\end{cases}
\end{equation}
This is the action of the discrete Laplacian
\begin{subequations}
\begin{equation}\label{La}
(\text{L}^{(m)}f)(l)= f(l+1)+f(l-1) \qquad  (l\in\Lambda^{(m)})
\end{equation} endowed
with three-point homogeneous boundary conditions at the lattice-ends:
\begin{equation}\label{Lb}
 f(-1):=-a_-f(1)+b_- f(0)   \quad\text{and}\quad  f(m+1):= -a_+ f(m-1) +b_+ f(m) ,
 \end{equation}
 \end{subequations}
which are governed by a total of four boundary parameters $a_-,b_-,a_+,b_+\in\mathbb{R}$.

When $(a_-,b_-)$ and $(a_+,b_+)$  belong to  $\{ (-1,0), (0,1)\}$ the boundary conditions in question are of
Neumann type (centered respectively at the boundary node or between the boundary node and the virtual  node on the exterior), whereas
for $(a_-,b_-)$ and $(a_+,b_+)$ taken from $\{ (0,0), (0,-1)\}$
one specializes to corresponding conditions of Dirichlet type.

\subsection{Diagonalization}
We will now diagonalize $\text{L}^{(m)}$ \eqref{L} in $C(\Lambda^{(m)})$ for parameters of the form
\begin{subequations}
\begin{equation}\label{parameters:a}
\boxed{a_\pm = p_\pm q_\pm \quad\text{and}\quad b_\pm = p_\pm + q_\pm ,} 
\end{equation}
with $p_-,q_-,p_+,q_+$ subject to the  restriction
\begin{equation}\label{parameters:b}
\boxed{-1< p_\pm \leq q_\pm < 1  .}
\end{equation}
\end{subequations}

Given $\hat{l}\in\Lambda^{(m)}$, let $\xi_{\hat{l}}^{(m)}$ denote the unique real-valued solution of the transcendental equation
\begin{subequations}
\begin{equation}\label{bethe:a}
2m \xi      + v_{p_-}(\xi )+v_{q_-}(\xi )+v_{p_+}(\xi )+v_{q_+}(\xi )       =  2\pi (\hat{l}+1),
\end{equation}
where for any $\xi\in\mathbb{R}$:
\begin{align}\label{bethe:b}
v_q(\xi ) &:= 
\int_0^\xi \frac{ (1-q^2)\ \text{d}x}{1-2q\cos(x)+q^2}
 \qquad (-1< q < 1) \\
 & =
i \log
\biggl( \frac{1- qe^{i\xi}}{e^{i\xi} - q }  \biggr) =  2\arctan \Biggl(    \frac{1+q}{1-q}\tan \left( \frac{\xi}{2} \right)   \Biggr).   \nonumber
\end{align}
\end{subequations} 
Here our choice for the branches of the logarithm and the arctangent are determined by the property that the real-analytic function $v_q(\xi)$ \eqref{bethe:b} is odd, strictly monotonously increasing, and quasi-periodic: $v_q(\xi  +2\pi)=v_q(\xi)+2\pi$. Since $v_q(0)=0$ and $v_q(\pi)=\pi$,
it is  manifest from Eq. \eqref{bethe:a} and the monotonicity of $v_q(\xi)$ \eqref{bethe:b} that
\begin{equation}\label{spectrum}
0<\xi_{0}^{(m)} <\xi_{1}^{(m)}< \cdots <\xi_{m-1}^{(m)}<\xi_{m}^{(m)}<\pi .
\end{equation}
The eigenbasis for $\text{L}^{(m)}$ turns out to be given by $(m+1)$ functions $\psi^{(m)}_{0} ,\ldots , \psi^{(m)}_{m}$ in $ C(\Lambda^{(m)})$
of the form
\begin{subequations}
\begin{equation}\label{e-vector:a}
\psi^{(m)}_{\hat{l}} (l) :=p_{l} (\xi^{(m)}_{\hat{l}}) \qquad (\hat{l}, l \in \Lambda^{(m)} ).
\end{equation}
Here $p_l (\xi)$ denotes the two-parameter Bernstein-Szeg\"o polynomial  (cf. e.g. \cite[Section 2.6]{sze:orthogonal}, \cite{gau-not:gauss} and \cite{gri:oscillatory})
\begin{align}\label{e-vector:b}
p_l (\xi)  & :=   c(\xi ) e^{i l\xi} + c(-\xi) e^{-il\xi} \quad\text{with}\quad c(\xi) := \frac{(1-p_-e^{-i\xi})(1-q_-e^{-i\xi})}{(1-e^{-2i\xi})}   \\
& = U_l(\cos(\xi))- b_-U_{l-1}(\cos(\xi)) + a_- U_{l-2}(\cos(\xi)) ,\nonumber
\end{align}
\end{subequations}
where $U_l(\cdot)$ refers to the Chebyshev polynomials of the second kind, i.e.
$
U_l(\cos(\xi)) := \sin\bigl((l+1)\xi\bigr)/\sin(\xi) 
$ ($l\in\mathbb{Z}$).

\begin{proposition}[Eigenfunctions]\label{eigenfunctions:prp}
(i) For any $\hat{l}\in\Lambda^{(m)}$, the function $\psi^{(m)}_{\hat{l}}$ \eqref{e-vector:a}, \eqref{e-vector:b}
satisfies the eigenvalue equation
\begin{equation}\label{eigenfunctions:eq}
\emph{L}^{(m)}\psi^{(m)}_{\hat{l}}=2\cos (\xi^{(m)}_{\hat{l}}) \psi^{(m)}_{\hat{l}} .
\end{equation}

(ii) The eigenfunctions  $\psi^{(m)}_{\hat{l}}$, $\hat{l}\in\Lambda^{(m)}$ constitute a basis for $C(\Lambda^{(m)})$.
\end{proposition}

\begin{proof}
Because $p_l(\xi )$ \eqref{e-vector:b}  is built from a linear superposition of the plane waves $e^{i\xi l}$ and $e^{-i\xi l}$,
it is clear that $p_{l+1}(\xi)+p_{l-1}(\xi)=2\cos (\xi) p_l (\xi)$ for all $l\in\mathbb{Z}$. To show that in fact Eq. \eqref{eigenfunctions:eq} holds,
it is enough to verify that at  $\xi=\xi^{(m)}_{\hat{l}}$ ($\hat{l}\in\Lambda^{(m)}$) the boundary conditions
\begin{equation*}
p_{-1}( \xi) = -a_-p_1(\xi)+b_-p_0(\xi ) \quad\text{and}\quad  p_{m+1}( \xi) = -a_+p_{m-1}(\xi)+b_+p_m(\xi )
\end{equation*}
are satisfied. Substitution of $p_l(\xi)=c(\xi) e^{i\xi l} + c(-\xi ) e^{-i\xi l}$ reformulates these  two boundary conditions in terms of functional relations for the complex amplitude $c(\xi)$:
\begin{equation*}
c(\xi) (1-p_- e^{i\xi})(1-q_-e^{i\xi})e^{-i\xi}
 + c(-\xi) (1-p_- e^{-i\xi})(1-q_-e^{-i\xi})e^{i\xi}=0
\end{equation*}
and
\begin{align*}
c(\xi) (1-p_+ e^{-i\xi}) & (1-q_+e^{-i\xi}) e^{i(m+1)\xi}  \,+ \\
& c(-\xi) (1-p_+ e^{i\xi})(1-q_+e^{i\xi}) e^{-i(m+1)\xi} =0,
\end{align*}
respectively.
Assuming $\xi$ real and $c(\xi)$ taken from Eq. \eqref{e-vector:b}, the first relation (and hence the first boundary condition) is seen to hold as a (trigonometric) polynomial identity for any value of $\xi$, whereas the second relation (and hence the second boundary condition) is  only satisfied  provided
\begin{equation}\label{bethe:eq}
e^{2im\xi}= \frac{(1-p_- e^{i\xi})(1-q_-e^{i\xi})(1-p_+ e^{i\xi})(1-q_+e^{i\xi})}{( e^{i\xi}-p_-)(e^{i\xi}-q_-)( e^{i\xi}-p_+)(e^{i\xi}-q_+)} .
\end{equation}
Upon multiplying Eq. \eqref{bethe:a} by the imaginary unit $i$ and exponentiating both sides with the aid of Eq. \eqref{bethe:b}, one confirms  that the algebraic relation in Eq. \eqref{bethe:eq} is  fulfilled at $\xi=\xi^{(m)}_{\hat{l}}$, $\hat{l}\in\Lambda^{(m)}$. This completes the proof of the statement that for any $\hat{l}\in\Lambda^{(m)}$ the function
$\psi_{\hat{l}}^{(m)}$ \eqref{e-vector:a}, \eqref{e-vector:b}
 solves the eigenvalue equation \eqref{eigenfunctions:eq}. The solutions in question give rise to
 nontrivial eigenfunctions of $\text{L}^{(m)}$ in $C(\Lambda^{(m)})$, because  $\psi_{\hat{l}}^{(m)}(l)\neq 0$ at $l=0$ as $p_0(\xi)=1-a_- >0$.
 Moreover, since the corresponding eigenvalues
$2\cos (\xi^{(m)}_0),\ldots ,2\cos (\xi^{(m)}_m)$ are distinct in view of Eq. \eqref{spectrum},
the eigenfunctions $\psi_{{0}}^{(m)},\ldots , \psi_{{m}}^{(m)}$ provide a basis for $C(\Lambda^{(m)})$.  \end{proof}

\begin{remark}
Since the derivative of $v_q(\xi)$ \eqref{bethe:b} remains bounded
\begin{equation*}
\frac{1-|q|}{1+|q|} \leq v_q^\prime (\xi)\leq \frac{1+|q|}{1-|q|} \quad \text{for}\  -1 <q <1, 
\end{equation*}
the subsequent estimates for the locations of the spectral points $\xi^{(m)}_{\hat{l}}$
and the size of the spectral gaps $ |   \xi^{(m)}_{\hat{l}} -  \xi^{(m)}_{\hat{k}}  |$  are immediate  from Eqs. \eqref{bethe:a}, \eqref{bethe:b}:
\begin{subequations}
\begin{equation}\label{e:momentgaps1}
\frac{\pi(\hat{l}+1)}{m+\kappa_-} \leq \xi^{(m)}_{\hat{l}} \leq \frac{\pi(\hat{l}+1)}{m+\kappa_+}\qquad (0\leq\hat{l}\leq m)
\end{equation}
and
\begin{equation}
\frac{\pi |\hat{l}-\hat{k}|}{m+\kappa_-} \leq  |   \xi^{(m)}_{\hat{l}} -  \xi^{(m)}_{\hat{k}}  | \leq \frac{\pi |\hat{l}-\hat{k}|}{m+\kappa_+}  \qquad (0\leq\hat{l}\neq \hat{k}\leq m)
,
\end{equation}
where
\begin{equation}
\kappa_{\pm} := 
\frac{1}{2}\left( \left(\frac{1-|p_-|}{1+ |p_-|}\right)^{\pm 1}+ \left(\frac{1-|q_-|}{1+ |q_-|}\right)^{\pm 1}+  \left(\frac{1-|p_+|}{1+ |p_+|}\right)^{\pm 1}+ \left(\frac{1-|q_+|}{1+ |q_+|}\right)^{\pm 1} \right).
\end{equation}
\end{subequations}
\end{remark}

\subsection{Plancherel formula}
For any $l\in\mathbb{Z}$, let 
\begin{equation*}
\delta_l:=\begin{cases} 1 &\text{if}\ l=0 ,\\  0 &\text{if}\ l\neq 0.
\end{cases}
\end{equation*}
The next proposition affirms that the eigenbasis in Proposition \ref{eigenfunctions:prp} is orthogonal with respect to the weights
\begin{equation}\label{weights}
\Delta^{(m)}_l :=(1-a_- \delta_l)^{-1} (1-a_+\delta_{m-l})^{-1}\qquad (l\in\Lambda^{(m)}),
\end{equation}
and in addition provides the quadratic norms turning this basis into an orthonormal one.

\begin{proposition}[Orthogonality Relations]\label{orthogonality:prp}
For parameters of the form in Eqs. \eqref{parameters:a}, \eqref{parameters:b}, the eigenfunctions  $\psi^{(m)}_{\hat{l}}$, $\hat{l}\in\Lambda^{(m)}$
satisfy the following orthogonality relations
\begin{subequations}
\begin{equation}\label{orthogonality:a}
\sum_{l\in\Lambda^{(m)}}  \psi^{(m)}_{\hat{l}}  (l) \psi^{(m)}_{\hat{k}} (l) \Delta^{(m)}_l =
\begin{cases}
0 & \text{if}\ \hat{k}\neq \hat{l} \\
1/ \hat{\Delta}^{(m)}_{\hat{l}} &\text{if}\ \hat{k}=\hat{l}
\end{cases} 
\end{equation}
($\hat{l},\hat{k}\in\Lambda^{(m)}$). Here the quadratic norms are governed by the Plancherel measure
\begin{equation}
 \hat{\Delta}^{(m)}_{\hat{l}} := \frac{1}{c(\xi^{(m)}_{\hat{l}} ) c(-\xi^{(m)}_{\hat{l}} )   H^{(m)}(\xi^{(m)}_{\hat{l}} )},
\end{equation}
with $c(\xi)$ taken from Eq. \eqref{e-vector:b} and
\begin{equation}
H^{(m)}(\xi):= 2m +   v_{p_-}^\prime (\xi )+v_{q_-}^\prime (\xi )+v_{p_+}^\prime (\xi )+v_{q_+}^\prime (\xi )  
\end{equation}
\end{subequations}
(where the prime indicates the derivative).
\end{proposition}

\begin{proof}
Let $\Delta^{(m)}$ denote the positive $(m+1)\times (m+1)$ diagonal matrix
\begin{equation}\label{gauge}
\Delta^{(m)}:= \text{diag} ( \Delta^{(m)}_0,  \Delta^{(m)}_1,\ldots , \Delta^{(m)}_m) .
 \end{equation}
Since the spectrum of $\text{L}^{(m)}$ \eqref{L} is simple by virtue of Proposition \ref{eigenfunctions:prp} and  Eq. \eqref{spectrum},
and the conjugated matrix $(\Delta^{(m)})^{1/2} \text{L}^{(m)} ( \Delta^{(m)})^{-1/2}$ is manifestly symmetric,
it is immediate that the
eigenbasis  $\psi_{{0}}^{(m)},\ldots , \psi_{{m}}^{(m)}$ satisfies the asserted orthogonality relations when $\hat{k}\neq \hat{l}$. On the other hand,
a straightforward  computation entails that
\begin{align*}
\sum_{l\in\Lambda^{(m)}}  p_l^2  (\xi) \Delta^{(m)}_l  =&
\sum_{l\in\Lambda^{(m)}} \left( c(\xi) e^{i\xi l}   +  c(-\xi ) e^{-i\xi l} \right)^2 \Delta^{(m)}_l \\
= &\,   c^2(\xi) \left( \frac{e^{2i\xi}-  e^{2im\xi }  }{1-e^{2i\xi}}  +  \frac{1}{1-a_-} +\frac{e^{2im \xi }}{1-a_+} \right) \\
&+  c^2(-\xi) \left( \frac{e^{-2i\xi}-  e^{-2im\xi }  }{1-e^{-2i\xi}}    +  \frac{1}{1-a_-} +\frac{e^{-2im\xi}}{1-a_+} \right) \\
& + 2 c(\xi ) c(-\xi)  \left(  m-1 +\frac{1}{1-a_-}+\frac{1}{1-a_+}        \right) .
\end{align*}
Upon recalling that at  $\xi=\xi_{\hat{l}}^{(m)}$  the algebraic relation in Eq. \eqref{bethe:eq} holds, we are in the position to rewrite
all instances of $e^{\pm 2im\xi }$  in terms of
\begin{equation*}
\left[ \frac{(1-p_- e^{i\xi})(1-q_-e^{i\xi})(1-p_+ e^{i\xi})(1-q_+e^{i\xi})}{( e^{i\xi}-p_-)(e^{i\xi}-q_-)( e^{i\xi}-p_+)(e^{i\xi}-q_+)} \right]^{\pm 1} .
\end{equation*}
This produces an expression for the quadratic norm that is readily seen to simplify
to $c(\xi)c(-\xi) H^{(m)}(\xi)$, therewith completing the proof of the asserted orthogonality relations when $\hat{k}= \hat{l}$.
\end{proof}

\begin{remark}\label{dual:rem}
Alternatively, Proposition \ref{orthogonality:prp} can be reformulated in terms of the following dual system of finite-dimensional orthogonality relations for the
Bernstein-Szeg\"o polynomials $p_0(\xi),\ldots ,p_m(\xi)$ supported on the nodes $\xi_{0}^{(m)},\ldots ,\xi_{m}^{(m)}$:
\begin{equation}
\sum_{\hat{l}\in\Lambda^{(m)}}   p_l(\xi_{\hat{l}}^{(m)}) p_k(\xi_{\hat{l}}^{(m)}) \hat{\Delta}_{\hat{l}}^{(m)}
=\begin{cases}
0 & \text{if}\ k\neq l \\
1/ \Delta^{(m)}_{l} &\text{if}\ k=l
\end{cases} 
\end{equation}
($l,k\in\Lambda^{(m)}$). 
\end{remark}

\subsection{Discrete Fourier transforms}\label{dft:sec}
Let $\ell^2(\Lambda^{(m)})$ denote the Hilbert space of $f\in C(\Lambda^{(m)})$ endowed with the standard inner product
\begin{equation}
\langle f ,g\rangle_{(m)}:= \sum_{l\in \Lambda^{(m)}}  f(l) \overline{g(l)} 
\end{equation}
(where the bar indicates the complex conjugate).
It is immediate from Proposition  \ref{orthogonality:prp} that the discrete Fourier transform $\mathcal{F}^{(m)}:\ell^2(\Lambda^{(m)})\to \ell^2(\Lambda^{(m)})$ with kernel
\begin{equation}
\Psi^{(m)}_{\hat{l},l} :=\sqrt{\hat{\Delta}_{\hat{l}}^{(m)}\Delta_l^{(m)}} \psi^{(m)}_{\hat{l}}(l) = \sqrt{\hat{\Delta}_{\hat{l}}^{(m)}\Delta_l^{(m)}}   p_l(\xi_{\hat{l}}^{(m)} )
\end{equation} 
($\hat{l},l\in\Lambda^{(m)}$)---given explicitly by the Fourier pairing
\begin{subequations}
\begin{equation}\label{dft:a}
\hat{f}({\hat{l}})= (\mathcal{F}^{(m)} f)({\hat{l}}) := \langle f , \Psi^{(m)}_{\hat{l},\cdot }\rangle_{(m)}= \sum_{l\in\Lambda^{(m)}} \Psi^{(m)}_{\hat{l},l } f(l) 
\end{equation}
($f\in \ell^2(\Lambda^{(m)})$, $ \hat{l}\in\Lambda^{(m)}$)---constitutes a unitary transformation in $\ell^2(\Lambda^{(m)})$ with an inversion formula of the form
\begin{equation}\label{dft:b}
f (l)=  ((\mathcal{F}^{(m)})^{-1} \hat{f})(l)=  \langle \hat{f} , \Psi^{(m)}_{\cdot ,l}\rangle_{(m)}= \sum_{\hat{l}\in\Lambda^{(m)}} \Psi^{(m)}_{\hat{l},l } \hat{f}(\hat{l}) 
\end{equation}
\end{subequations}
($\hat{f}\in \ell^2(\Lambda^{(m)})$, ${l}\in\Lambda^{(m)}$). (In these pairings the subscript dots in $ \Psi^{(m)}_{\hat{l},\cdot }$ and $\Psi^{(m)}_{\cdot ,l}$  indicate the slots corresponding to the variable.)

\begin{table}[]
\centering
\caption{Discrete (co)sine transforms corresponding to boundary conditions of Neumann and Dirichlet type}
\label{DCT-DST:table}
\begin{tabular}{@{}lc|cccc|c@{}}
\toprule
                                        &        & \multicolumn{2}{c|}{N}                                   & \multicolumn{2}{c|}{D}              & \multicolumn{1}{l}{}        \\ \midrule
                                        & $(p_\pm,q_\pm)$  & \multicolumn{1}{c|}{(-1,1)} & \multicolumn{1}{c|}{(0,1)} & \multicolumn{1}{c|}{(0,0)} & (-1,0) & +/-                         \\ \midrule
\multicolumn{1}{l|}{\multirow{2}{*}{\vspace{-1ex} N}} & (-1,1) & DCT-1                       & DCT-5                      & DCT-3                      & DCT-7  & \multicolumn{1}{c|}{(-1,0)} \\ \cmidrule(lr){2-2} \cmidrule(l){7-7} 
\multicolumn{1}{l|}{}                   & (0,1)  & DCT-6                       & DCT-2                      & DCT-8                      & DCT-4  & \multicolumn{1}{c|}{(0,1)}  \\ \cmidrule(r){1-2} \cmidrule(l){7-7} 
\multicolumn{1}{l|}{\multirow{2}{*}{\vspace{-1ex} D}} & (0,0)  & DST-3                       & DST-7                      & DST-1                      & DST-5  & \multicolumn{1}{c|}{(0,0)}  \\ \cmidrule(lr){2-2} \cmidrule(l){7-7} 
\multicolumn{1}{l|}{}                   & (-1,0) & DST-8                       & DST-4                      & DST-6                      & DST-2  & \multicolumn{1}{c|}{(0,-1)} \\ \midrule
                                        & -/+    & \multicolumn{1}{c|}{(-1,0)} & \multicolumn{1}{c|}{(0,1)} & \multicolumn{1}{c|}{(0,0)} & (0,-1) & $(a_\pm,b_\pm) $                       \\ \bottomrule
\end{tabular}
\end{table}

By performing the limits to the boundary parameters values encoding Neumann and Dirichlet type boundary conditions in accordance with Table \ref{DCT-DST:table}, the
discrete Fourier transform in Eqs. \eqref{dft:a}, \eqref{dft:b} is seen to degenerate into the sixteen standard  discrete (co)sine transforms
DCT-$k$ and DST-$k$ ($k=1,\ldots ,8$) \cite{bri-yip-rao:discrete}.

\begin{itemize}
\item[DCT-1:] $(p_\pm,q_\pm)\to (-1,1)\Rightarrow$
 \begin{equation*}\textstyle
 \xi_{\hat{l}}^{(m)}\to  \frac{\pi\hat{l}}{m} \, , \qquad
  \Psi_{\hat{l},l}^{(m)}\to \sqrt{\frac{2}{m}}\cos \bigl(\frac{\pi\hat{l}l}{m}\bigr) \bigl(\frac{1}{\sqrt{2}}\bigr)^{\delta_{\hat{l}}+\delta_{m-\hat{l}}+\delta_l+\delta_{m-l}}
\end{equation*}

\item[DCT-2:]
$(p_\pm,q_\pm)\to (0,1)\Rightarrow$
 \begin{equation*}\textstyle
 \xi_{\hat{l}}^{(m)}\to  \frac{\pi\hat{l}}{m+1}\, , \qquad
  \Psi_{\hat{l},l}^{(m)}\to \sqrt{\frac{2}{m+1}}\cos \bigl(\frac{\pi\hat{l}(l+\frac{1}{2})}{m+1}\bigr) \bigl(\frac{1}{\sqrt{2}}\bigr)^{\delta_{\hat{l}}}
\end{equation*}

\item[DCT-3:] $(p_-,q_-)\to (-1,1)$ and $(p_+,q_+)\to (0,0)\Rightarrow$
 \begin{equation*}\textstyle
 \xi_{\hat{l}}^{(m)}\to  \frac{\pi(\hat{l}+\frac12)}{m+1}\, , \qquad
  \Psi_{\hat{l},l}^{(m)}\to \sqrt{\frac{2}{m+1}}\cos \bigl(\frac{\pi (\hat{l}+\frac{1}{2})l}{m+1}\bigr) \bigl(\frac{1}{\sqrt{2}}\bigr)^{\delta_{l}}
\end{equation*}

\item[DCT-4:] $(p_-,q_-)\to (0,1)$ and $(p_+,q_+)\to (-1,0)\Rightarrow$
 \begin{equation*}\textstyle
 \xi_{\hat{l}}^{(m)}\to  \frac{\pi(\hat{l}+\frac12)}{m+1}\, , \qquad
  \Psi_{\hat{l},l}^{(m)}\to \sqrt{\frac{2}{m+1}}\cos \bigl(\frac{\pi (\hat l+\frac12)(l+\frac{1}{2})}{m+1}\bigr)
\end{equation*}

\item[DCT-5:] $(p_-,q_-)\to (-1,1)$ and $(p_+,q_+)\to (0,1)\Rightarrow$

 \begin{equation*}\textstyle
 \xi_{\hat{l}}^{(m)}\to  \frac{\pi \hat{l}}{m+\frac12}\, , \qquad
  \Psi_{\hat{l},l}^{(m)}\to \sqrt{\frac{2}{m+\frac12}}\cos \bigl(\frac{\pi l \hat{l}}{m+\frac12}\bigr)
      \bigl(\frac{1}{\sqrt{2}}\bigr)^{\delta_{l}+\delta_{\hat l}}
\end{equation*}

\item[DCT-6:] $(p_-,q_-)\to (0,1)$ and $(p_+,q_+)\to (-1,1)\Rightarrow$

 \begin{equation*}\textstyle
 \xi_{\hat{l}}^{(m)}\to  \frac{\pi \hat{l}}{m+\frac12}
     \, , \qquad
  \Psi_{\hat{l},l}^{(m)}\to \sqrt{\frac{2}{m+\frac12}}\cos \bigl(\frac{\pi  \hat{l} (l+\frac12)}{m+\frac12}\bigr)
      \bigl(\frac{1}{\sqrt{2}}\bigr)^{\delta_{\hat l}+\delta_{m- l}}
\end{equation*}

\item[DCT-7:] $(p_-,q_-)\to (-1,1)$ and $(p_+,q_+)\to (-1,0)\Rightarrow$

 \begin{equation*}\textstyle
 \xi_{\hat{l}}^{(m)}\to  \frac{\pi (\hat{l} +\frac12)}{m+\frac12}
     \, , \qquad
  \Psi_{\hat{l},l}^{(m)}\to \sqrt{\frac{2}{m+\frac12}}\cos \bigl(\frac{\pi   (\hat l+\frac12) l }{m+\frac12}\bigr)
     \bigl(\frac{1}{\sqrt{2}}\bigr)^{\delta_{m-\hat l}+\delta_{l}}
\end{equation*}

\item[DCT-8:] $(p_-,q_-)\to (0,1)$ and $(p_+,q_+)\to (0,0)\Rightarrow$

 \begin{equation*}\textstyle
 \xi_{\hat{l}}^{(m)}\to  \frac{\pi (\hat{l} +\frac12)}{m+\frac32}
     \, , \qquad
  \Psi_{\hat{l},l}^{(m)}\to \sqrt{\frac{2}{m+\frac32}}\cos \bigl(\frac{\pi   (\hat l+\frac12) (l+\frac12) }{m+\frac32}\bigr)
\end{equation*}

\end{itemize}

\begin{itemize}
\item[DST-1:] $(p_\pm,q_\pm)\to (0,0)\Rightarrow$

\begin{equation*}\textstyle
 \xi_{\hat{l}}^{(m)}\to  \frac{\pi (\hat l + 1)}{m+2}  
   \, , \qquad
  \Psi_{\hat{l},l}^{(m)}\to \sqrt{\frac{2}{m+2}}\sin \bigl(\frac{\pi (\hat l + 1)(l+1)}{m+2}\bigr)
\end{equation*}

\item[DST-2:] $(p_\pm,q_\pm)\to (-1,0)\Rightarrow$

\begin{equation*}\textstyle
 \xi_{\hat{l}}^{(m)}\to  \frac{\pi (\hat l + 1)}{m+1}  
   \, , \qquad
  \Psi_{\hat{l},l}^{(m)}\to \sqrt{\frac{2}{m+1}}\sin \bigl(\frac{\pi (\hat l + 1)(l+\frac12)}{m+1}\bigr)
        \bigl(\frac{1}{\sqrt{2}}\bigr)^{\delta_{m-\hat l}}
\end{equation*}

\item[DST-3:] $(p_-,q_-)\to (0,0)$ and $(p_+,q_+)\to (-1,1)\Rightarrow$

\begin{equation*}\textstyle
 \xi_{\hat{l}}^{(m)}\to  \frac{\pi (\hat l + \frac12)}{m+1}  
   \, , \qquad
  \Psi_{\hat{l},l}^{(m)}\to \sqrt{\frac{2}{m+1}}\sin \bigl(\frac{\pi (\hat l + \frac12)(l+1)}{m+1}\bigr)
          \bigl(\frac{1}{\sqrt{2}}\bigr)^{\delta_{m-l} }
\end{equation*}

\item[DST-4:] $(p_-,q_-)\to (-1,0)$ and $(p_+,q_+)\to (0,1)\Rightarrow$

\begin{equation*}\textstyle
 \xi_{\hat{l}}^{(m)}\to  \frac{\pi (\hat l + \frac12)}{m+1}  
   \, , \qquad
  \Psi_{\hat{l},l}^{(m)}\to \sqrt{\frac{2}{m+1}}\sin \bigl(\frac{\pi (\hat l + \frac12)(l+\frac12)}{m+1}\bigr)
\end{equation*}

\item[DST-5:] $(p_-,q_-)\to (0,0)$ and $(p_+,q_+)\to (-1,0)\Rightarrow$

\begin{equation*}\textstyle
 \xi_{\hat{l}}^{(m)}\to  \frac{\pi (\hat l + 1)}{m+\frac32}  
   \, , \qquad
  \Psi_{\hat{l},l}^{(m)}\to \sqrt{\frac{2}{m+\frac32}}\sin \bigl(\frac{\pi (\hat l + 1)(l+1)}{m+\frac32}\bigr)
\end{equation*}

\item[DST-6:] $(p_-,q_-)\to (-1,0)$ and $(p_+,q_+)\to (0,0)\Rightarrow$

\begin{equation*}\textstyle
 \xi_{\hat{l}}^{(m)}\to  \frac{\pi (\hat l + 1)}{m+\frac32}  
   \, , \qquad
  \Psi_{\hat{l},l}^{(m)}\to \sqrt{\frac{2}{m+\frac32}}\sin \bigl(\frac{\pi (\hat l + 1)(l+\frac12)}{m+\frac32}\bigr)
\end{equation*}

\item[DST-7:] $(p_-,q_-)\to (0,0)$ and $(p_+,q_+)\to (0,1)\Rightarrow$

\begin{equation*}\textstyle
 \xi_{\hat{l}}^{(m)}\to  \frac{\pi (\hat l + \frac12)}{m+\frac32}  
   \, , \qquad
  \Psi_{\hat{l},l}^{(m)}\to \sqrt{\frac{2}{m+\frac32}}\sin \bigl(\frac{\pi (\hat l + \frac12)(l+1)}{m+\frac32}\bigr)
\end{equation*}

\item[DST-8:] $(p_-,q_-)\to (-1,0)$ and $(p_+,q_+)\to (-1,1)\Rightarrow$

\begin{equation*}\textstyle
 \xi_{\hat{l}}^{(m)}\to  \frac{\pi (\hat l + \frac12)}{m+\frac12}  
   \, , \qquad
  \Psi_{\hat{l},l}^{(m)}\to \sqrt{\frac{2}{m+\frac12}}\sin \bigl(\frac{\pi (\hat l + \frac12)(l+\frac12)}{m+\frac12}\bigr)
        \bigl(\frac{1}{\sqrt{2}}\bigr)^{\delta_{m-\hat l} + \delta_{m-l} }
\end{equation*}
\end{itemize}

To verify the above limit transitions, one first computes the limiting values of the spectral points
$\xi_{\hat{l}}^{(m)}$ by means of Eq.  \eqref{bethe:a}, cf. Remark \ref{limit-to-ND-spectrum:rem} below for some further details. The limits of the eigenfunctions  $\psi_{\hat{l}}^{(m)}(l)$ then follow via the representation  of $p_l(\xi)$ \eqref{e-vector:b} in terms of Chebyshev polynomials. To recover the limits of the Fourier kernel
 $\Psi_{\hat{l},l}^{(m)}$ it remains to compute the limits of the weights
 $\Delta^{(m)}_l$ and $\hat\Delta^{(m)}_{\hat l}$. While for  $\Delta^{(m)}_l$ this is trivial, for $\hat\Delta^{(m)}_{\hat l}$  the limit in question is straightforward from the explicit expressions only when the limiting value of
 $\xi_{\hat{l}}^{(m)}$ amounts  to an interior point of  the interval $[0,\pi]$.  On the other hand, if $\xi_{\hat{l}}^{(m)}$ converges to a boundary point then $\bigr(\psi_{\hat{l}}^{(m)}(l)\bigl)^2$ converges to a constant function, whence the limiting behavior of $\hat\Delta^{(m)}_{\hat l}$ is plain from Eq. \eqref{orthogonality:a} in this situation.

\begin{remark}\label{limit-to-ND-spectrum:rem}
To compute the limiting values of the spectral points $\xi^{(m)}_{\hat{l}}$ corresponding to the values of the boundary parameters in Table \ref{DCT-DST:table},
 one uses that for $0<\xi <\pi$:
 \begin{equation}\label{limit-v}
 \lim_{q\to \varepsilon} v_q(\xi)=
 \begin{cases}   
0 & \text{if}\  \varepsilon=-1, \\
 \xi & \text{if}\  \varepsilon=0 ,\\
 \pi & \text{if}\  \varepsilon =1 
 \end{cases} 
 \end{equation}
 (uniformly on compacts). Indeed, it is read-off from Eq. \eqref{bethe:a} that (i)
 $\xi^{(m)}_{\hat l}\to 0$  if $\hat l=0$ and $q_-,q_+\to 1$,  (ii) $\xi^{(m)}_{\hat l}\to \pi$  if $\hat l=m$ and $p_-,p_+\to -1$, and (iii)
 that there exists an $\epsilon >0$ (depending only on $m$) such that 
 $ \epsilon \le \xi^{(m)}_{\hat l} \le \pi-\epsilon
 $
 in all other cases.
We thus conclude (from (i),(ii), and (iii) in combination with Eq. \eqref{bethe:a} and the locally uniform convergence \eqref{limit-v}) that
$$
\xi^{(m)}_{\hat l} \to \frac{\pi (\hat l +1 -\frac{1}{2}N_1)}{m +\frac{1}{2}N_0}  ,
$$
where $N_0$ and $N_1$ indicate the number of the parameters $p_-,q_-,p_+,q_+$ that converge to $0$ and $1$, respectively.
\end{remark}

\section{Multivariate generalization via generalized Schur polynomials}\label{sec3}
Upon identifying
Macdonald's ninth variation of the Schur polynomials
\cite{mac:schur,nak-nou-shi-yam:tableau,ser-ves:jacobi} associated with the two-parameter Bernstein-Szeg\"o family in Eq. \eqref{e-vector:b}
as a $t\to 0$ parameter degeneration of
Macdonald's three-parameter hyperoctahedral Hall-Littlewood polynomials associated with the root system $BC_n$
\cite[\S 10]{mac:orthogonal},
we arrive at a multivariate generalization of the discrete Fourier transform in Section \ref{sec2}. 
For the special parameter values corresponding to the standard boundary conditions of Dirichlet- and Neumann type, the construction then produces
multivariate generalizations of the pertinent discrete (co)sine transforms.
The latter transforms belong to a much wider class of multivariate discrete (co)sine transforms that was studied systematically by
 Klimyk, Moody and Patera et al. \cite{kli-pat:anti-b,moo-pat:cubature,moo-mot-pat:gaussian,czy-hri:generalized,hri-mot:discrete}  within the framework of (affine) root systems.
 From this perspective,
 the discrete (co)sine transforms emerging here pertain to the root system $BC_n$.

\subsection{Generalized Schur polynomials}
For any $\lambda=(\lambda_1,\ldots ,\lambda_n)\in\mathbb{Z}^n$ with weakly decreasing nonnegative parts
\begin{equation*}
\lambda_1\geq\lambda_2\geq\cdots \geq\lambda_n\geq 0,
\end{equation*}
the generalized Schur polynomial $P_\lambda(\boldsymbol{\xi}) $ in $\boldsymbol{\xi}:=(\xi_1,\ldots,\xi_n)$ associated with the two-parameter Bernstein-Szeg\"o family $p_l(\xi)$ \eqref{e-vector:b} is defined via the determinantal formula
\begin{equation}\label{Schur:det}
P_\lambda(\boldsymbol{\xi}) := \frac{1}{V(\boldsymbol{\xi})} \det [ p_{n-j+\lambda_j}(\xi_k)]_{1\leq j,k\leq n} ,
\end{equation}
where $V(\boldsymbol{\xi})$ refers to the Vandermonde determinant
\begin{equation*}
V(\boldsymbol{\xi})=\prod_{1\leq j <k \leq n}   2\bigl( \cos(\xi_j)-\cos(\xi_k) \bigr) .
\end{equation*}
Notice that if we replace $p_l (\xi)$ by $ z^l$ with $z=2\cos(\xi)$ on the RHS, then Eq. \eqref{Schur:det} reproduces precisely the celebrated determinantal representation defining the
classical Schur polynomial
$s_\lambda(z_1,\ldots ,z_n )$ in variables $z_j=2\cos(\xi_j)$, $j=1,\ldots ,n$ \cite{mac:schur}. Moreover, by expanding the determinant in the numerator of Eq. \eqref{Schur:det} it is seen that $P_\lambda (\boldsymbol{\xi})$ amounts to the following multivariate version of the Bernstein-Szeg\"o polynomial $p_l(\xi)$ \eqref{e-vector:b}:
\begin{subequations}
\begin{equation}\label{Schur:exp}
P_\lambda(\boldsymbol{\xi}) =
 \sum_{( \sigma,\epsilon)\in S_n\ltimes \{ 1,-1\}^n}   C(\epsilon_1 \xi_{\sigma_1},\ldots , \epsilon_n \xi_{\sigma_n})
\exp (i\epsilon_1 \xi_{\sigma_1}\lambda_1+\cdots +i \epsilon_n \xi_{\sigma_n} \lambda_n) 
\end{equation}
with
\begin{eqnarray}\label{Cf}
\lefteqn{C(\xi_1,\ldots ,\xi_n)= C(\boldsymbol{\xi}) :=
\prod_{1\leq j\leq n} \frac{(1-p_- e^{-i\xi_j})( 1- q_- e^{-i\xi_j})}{1-e^{-2i\xi_j}}} && \\
&& \times \prod_{1\leq j<k \leq n} (1-e^{-i(\xi_{j}+\xi_k)})^{-1} (1-e^{-i(\xi_{j}-\xi_k)})^{-1}.\nonumber
\end{eqnarray}
\end{subequations}
In this explicit representation the summation is meant over all signed permutations $(\sigma,\epsilon)$ with $\sigma= { \bigl( \begin{smallmatrix}1& 2& \cdots & n \\
 \sigma_1&\sigma_2&\cdots & \sigma_n
 \end{smallmatrix}\bigr)}$ belonging to the symmetric group $S_n$ and 
$\epsilon=(\epsilon_1,\ldots,\epsilon_n)\in \{ 1,-1\}^n$.
It is evident from Eqs. \eqref{Schur:exp}, \eqref{Cf} that the two-parameter generalized Schur polynomial $P_\lambda (\boldsymbol{\xi})$  \eqref{Schur:det} boils down to a parameter specialization (viz. $t\to 0$) of Macdonald's three-parameter hyperoctahedral Hall-Littlewood polynomial associated with the root system $BC_n$ \cite[\S 10]{mac:orthogonal}.

\subsection{Plancherel formula}
To any $\hat{\lambda}$ belonging to
\begin{equation}\label{FD}
\Lambda^{(m,n)}:=\{  \lambda=(\lambda_1,\ldots,\lambda_n) \in\mathbb{Z}^n \mid m\geq \lambda_1\geq\lambda_2\geq\cdots \geq\lambda_n\geq 0\} 
\end{equation}
we now associate a lattice function  $\psi^{(m,n)}_{\hat{\lambda}}$ in the space $ C(\Lambda^{(m,n)})$ of functions $f:\Lambda^{(m,n)}\to\mathbb{C}$, which is given by
\begin{subequations}
\begin{equation}\label{e-vector-n:a}
\psi^{(m,n)}_{\hat{\lambda}}(\lambda):= P_\lambda \bigl( \boldsymbol{\xi}^{(m,n)}_{\hat{\lambda}}\bigr)
\end{equation}
($\lambda\in\Lambda^{(m,n)}$) with
\begin{equation}\label{e-vector-n:b}
\boldsymbol{\xi}^{(m,n)}_{\hat{\lambda}}:=\left(\xi^{(m+n-1)}_{\hat{\lambda}_1+n-1},\xi^{(m+n-1)}_{\hat{\lambda}_2+n-2},\ldots ,\xi^{(m+n-1)}_{\hat{\lambda}_{n-1}+1},\xi^{(m+n-1)}_{\hat{\lambda}_n}\right)
\end{equation}
\end{subequations}
($\hat{\lambda}\in \Lambda^{(m,n)}$).
The following theorem---which reveals that the functions $\psi^{(m,n)}_{\hat{\lambda}}$, $\hat{\lambda}\in \Lambda^{(m,n)}$ constitute an orthogonal basis
of $C(\Lambda^{(m,n)})$ with respect to the weights 
\begin{equation}\label{weights-n}
\Delta^{(m,n)}_\lambda :=\prod_{1\leq j \leq n} \Delta^{(m+n-1)}_{n-j+\lambda_j}=
(1-a_- \delta_{\lambda_n})^{-1} (1-a_+\delta_{m-\lambda_1})^{-1}
\end{equation}
($\lambda\in \Lambda^{(m,n)}$), and which in addition computes the corresponding Plancherel measure explicitly---generalizes
Proposition \ref{orthogonality:prp} to the situation of an arbitrary number of variables $n$.

\begin{theorem}[Orthogonality Relations]\label{orthogonality-n:thm}
For any $\hat{\lambda},\hat{\mu}\in\Lambda^{(m,n)}$ and parameters of the form Eqs. \eqref{parameters:a}, \eqref{parameters:b}, one has that
\begin{subequations}
\begin{equation}\label{orhogonality-n:a}
\sum_{\lambda\in\Lambda^{(m,n)}}  \psi^{(m,n)}_{\hat{\lambda}}  (\lambda) \psi^{(m,n)}_{\hat{\mu}} (\lambda) \Delta^{(m,n)}_\lambda =
\begin{cases}
0 & \text{if}\ \hat{\mu}\neq \hat{\lambda} \\
1/ \hat{\Delta}^{(m,n)}_{\hat{\lambda}} &\text{if}\ \hat{\mu}=\hat{\lambda}
\end{cases} ,
\end{equation}
where
\begin{equation}
 \hat{\Delta}^{(m,n)}_{\hat{\lambda}} := \frac{1}{C\bigl( \boldsymbol{\xi}^{(m,n)}_{\hat{\lambda}} \bigr) C\bigl(-\ \boldsymbol{\xi}^{(m,n)}_{\hat{\lambda}} \bigr)   H^{(m,n)}\bigl( \boldsymbol{\xi}^{(m,n)}_{\hat{\lambda}} \bigr)},
\end{equation}
with $C(\boldsymbol{\xi})$ taken from Eq. \eqref{Cf} and
\begin{equation}
H^{(m,n)}(\boldsymbol{\xi}):= \prod_{1\leq j\leq n} H^{(m+n-1)}(\xi_j) .
\end{equation}
\end{subequations}
\end{theorem}
The proof of these orthogonality relations---the details of which are relegated to Section \ref{orthogonality-n:prf} below---hinges on the Cauchy-Binet formula.

\begin{remark}\label{factorization-plancherel-measure:rem}
The weights of the Plancherel measure in Theorem \ref{orthogonality-n:thm} admit a factorization of the form
\begin{equation}\label{factorization-plancherel-measure:eq}
\hat{\Delta}^{(m,n)}_{\hat{\lambda}}= V(\boldsymbol{\xi}^{(m,n)}_{\hat{\lambda}} )^2 \prod_{1\leq j\leq n}  \hat{\Delta}^{(m+n-1)}_{n-j+\hat{\lambda}_j}
\end{equation}
($\hat{\lambda}\in \Lambda^{(m,n)}$).
\end{remark}

\begin{remark}\label{orthogonality-n:rem} The corresponding dual description of Theorem \ref{orthogonality-n:thm}---extending Remark \ref{dual:rem} to the case $n>1$---is encoded by the
following
finite-dimensional system of
orthogonality relations for $P_\lambda(\boldsymbol{\xi})$, $\lambda\in \Lambda^{(m,n)}$ supported at the nodes $\boldsymbol{\xi}^{(m,n)}_{\hat{\lambda}}$, $\hat{\lambda}\in\Lambda^{(m,n)}$:
\begin{equation}
\sum_{\hat{\lambda}\in\Lambda^{(m,n)}}   P_\lambda (\boldsymbol{\xi}_{\hat{\lambda}}^{(m,n)}) P_\mu(\boldsymbol{\xi}_{\hat{\lambda}}^{(m,n)}) \hat{\Delta}_{\hat{\lambda}}^{(m,n)}
=\begin{cases}
0 & \text{if}\ \mu\neq \lambda \\
1/ \Delta^{(m,n)}_{\lambda} &\text{if}\ \mu=\lambda
\end{cases} 
\end{equation}
($\lambda,\mu\in\Lambda^{(m,n)}$). 
\end{remark}

\begin{remark} It follows from \cite[Rem. 3.7]{die-ems-zur:completeness}  that $\psi^{(m,n)}_{\hat{\lambda}}$ \eqref{e-vector-n:a}, \eqref{e-vector-n:b}
obeys the following eigenvalue equation generalizing Eq. \eqref{eigenfunctions:eq}:
\begin{subequations}
\begin{equation}
\text{L}^{(m,n)} \psi^{(m,n)}_{\hat{\lambda}}= E \bigl(\boldsymbol{\xi}^{(m,n)}_{\hat{\lambda}}\bigr) \psi^{(m,n)}_{\hat{\lambda}} \qquad (\hat{\lambda}\in\Lambda^{(m,n)}).
\end{equation}
Here $\text{L}^{(m,n)}$  denotes a discrete Laplacian whose action on $f\in C(\Lambda^{(m,n)})$
evaluated at $\lambda\in\Lambda^{(m,n)}$ is given by
\begin{align}\label{Lmn}
& (\text{L}^{(m,n)} f)(\lambda)
= \Bigl( b_-\delta_{\lambda_n} +
b_+\delta_{m-\lambda_{1}}\Bigr)
f(\lambda)  +\\
&\sum_{\substack{1\leq j \leq n\\ \lambda+e_j\in\Lambda^{(m,n)}}}
(1-a_-\delta_{\lambda_j} )^{\delta_{n-j}}
f(\lambda+e_j) 
+\sum_{\substack{1\leq j \leq n\\ \lambda-e_j\in\Lambda^{(m,n)}}}
(1-a_+\delta_{m-\lambda_j} )^{\delta_{j-1}}
 f(\lambda-e_j) \nonumber
\end{align} 
(with the vectors $e_1,\ldots ,e_n$ representing the standard unit basis of $\mathbb{Z}^n$), while its eigenvalues are governed by
\begin{equation}
E (\boldsymbol{\xi}):= \sum_{1\leq j\leq n} 2\cos(\xi_j) .
\end{equation}
\end{subequations}
In other words, 
the orthogonal basis $\psi^{(m,n)}_{\hat{\lambda}}$, $\hat{\lambda}\in\Lambda^{(m,n)}$ diagonalizes the self-adjoint Laplacian $\text{L}^{(m,n)}$ \eqref{Lmn}
in the $\ell^2$-space over $\Lambda^{(m,n)}$ \eqref{FD} determined by the  weights $\Delta^{(m,n)}_\lambda$ \eqref{weights-n}.
 The  Laplacian at issue has its origin in a mathematical physics context as the $n$-particle Hamiltonian for
a phase model describing strongly correlated bosons on the finite one-dimensional aperiodic lattice $\Lambda^{(m)}$ endowed with open-end boundary interactions \cite[\text{Rem.} 2.2]{die-ems-zur:completeness}.  Previously, related quantum Hamiltonians modeling analogous bosonic $n$-particle systems on the  periodic one-dimensional lattice $\mathbb{Z}/m\mathbb{Z}$ were shown to be diagonalizable by means of  standard Schur polynomials $s_\lambda (z_1,\ldots ,z_n)$
\cite{bog:boxed,die:diagonalization,kor-stro:slnk}.
\end{remark}

\subsection{Discrete Fourier transforms}
By building the normalized kernel
\begin{align}\label{fourier-kernel-n}
\Psi^{(m,n)}_{\hat{\lambda},\lambda } :=& \sqrt{\hat{\Delta}_{\hat{\lambda}}^{(m,n)}\Delta_\lambda^{(m,n)}} \psi^{(m,n)}_{\hat{\lambda}}(\lambda ) = \sqrt{\hat{\Delta}_{\hat{\lambda}}^{(m,n)}\Delta_\lambda^{(m,n)}}
 P_\lambda\bigl(\boldsymbol{\xi}_{\hat{\lambda}}^{(m,n)} \bigr) \\
 =& \det \left[ \Psi^{(m+n-1)}_{n-j+\hat{\lambda}_j,n-k+\lambda_k}\right]_{1\leq j,k\leq n} \nonumber
\end{align} 
($\hat{\lambda},\lambda\in\Lambda^{(m,n)}$), one is  led to a unitary multivariate discrete Fourier transform $\mathcal{F}^{(m,n)}:\ell^2(\Lambda^{(m,n)})\to \ell^2(\Lambda^{(m,n)})$ 
in the Hilbert space
$\ell^2(\Lambda^{(m,n)})$  of $f\in C(\Lambda^{(m,n)})$  with the standard inner product
\begin{equation}
\langle f ,g\rangle_{(m,n)}:= \sum_{\lambda\in \Lambda^{(m,n)}}  f(\lambda) \overline{g(\lambda)} .
\end{equation}
Specifically, we thus obtain the following multivariate generalization of the Fourier pairing in Eq. \eqref{dft:a}
\begin{subequations}
\begin{equation}\label{dft-n:a}
\hat{f}({\hat{\lambda}})= (\mathcal{F}^{(m,n)} f)({\hat{\lambda}}) := \langle f , \Psi^{(m,n)}_{\hat{\lambda},\cdot }\rangle_{(m,n)}= \sum_{\lambda\in\Lambda^{(m,n)}} \Psi^{(m,n)}_{\hat{\lambda},\lambda} f(\lambda) 
\end{equation}
($f\in \ell^2(\Lambda^{(m,n)})$, $ \hat{\lambda}\in\Lambda^{(m,n)}$), and of its inversion formula  in Eq. \eqref{dft:b}
\begin{equation}\label{dft-n:b}
f (\lambda)=  ((\mathcal{F}^{(m,n)})^{-1} \hat{f})(\lambda)=  \langle \hat{f} , \Psi^{(m,n)}_{\cdot ,\lambda}\rangle_{(m,n)}= \sum_{\hat{\lambda}\in\Lambda^{(m,n)}} \Psi^{(m,n)}_{\hat{\lambda},\lambda} \hat{f}(\hat{\lambda}) 
\end{equation}
\end{subequations}
($\hat{f}\in \ell^2(\Lambda^{(m,n)})$, $\lambda\in\Lambda^{(m,n)}$).  Upon degenerating the kernel in the Slater determinant on the second line of Eq. \eqref{fourier-kernel-n} to the parameter values pertaining to boundary conditions of Dirichlet and Neumann type as detailed in Section \ref{dft:sec}, one reproduces antisymmetric multivariate counterparts
of the DST-1,$\ldots$,DST-8 and DCT-1,$\ldots$,DCT-8 that belong to the families studied
in Refs. \cite{kli-pat:anti-b,moo-pat:cubature,moo-mot-pat:gaussian,czy-hri:generalized,hri-mot:discrete}.

\subsection{Proof of Theorem \ref{orthogonality-n:thm}}\label{orthogonality-n:prf}
Rather than to verify the orthogonality relations of Theorem \ref{orthogonality-n:thm} directly, we will instead prove the equivalent (by `column-row duality') orthogonality relations
of Remark \ref{orthogonality-n:rem}.
Since it is clear from the definitions that
\begin{align*}
  \hat{\Delta}_{\hat{\lambda}}^{(m,n)}  P_\lambda (\boldsymbol{\xi}_{\hat{\lambda}}^{(m,n)}) & P_\mu(\boldsymbol{\xi}_{\hat{\lambda}}^{(m,n)})  =
 \prod_{1\leq j\leq n}  \hat{\Delta}^{(m+n-1)}_{n-j+\hat{\lambda}_j}  
\\
\times 
 \det & \left[ p_{n-j+\lambda_j}\bigl(\xi^{(m+n-1)}_{n-k+\hat{\lambda}_k}\bigr)\right]_{1\leq j,k\leq n}   \det \left[ p_{n-j+\mu_j}\bigl(\xi^{(m+n-1)}_{n-k+\hat{\lambda}_k}\bigr)\right]_{1\leq j,k\leq n}   \\
=    \frac{1}{ \sqrt{ \Delta^{(m,n)}_\lambda   \Delta^{(m,n)}_\mu }     }&
   \det \left[ \Psi^{(m+n-1)}_{n-j+\hat{\lambda}_j,n-k+\lambda_k}\right]_{1\leq j,k\leq n}   \det \left[ \Psi^{(m+n-1)}_{n-j+\hat{\lambda}_j,n-k+\mu_k}\right]_{1\leq j,k\leq n} 
\end{align*}
(cf.  Eqs. \eqref{Schur:det},  \eqref{weights-n}, Remark \ref{factorization-plancherel-measure:rem}, and Eq. \eqref{fourier-kernel-n}), one has that
\begin{align*}
&\sqrt{ \Delta^{(m,n)}_\lambda   \Delta^{(m,n)}_\mu }     \sum_{\hat{\lambda}\in\Lambda^{(m,n)}}
P_\lambda (\boldsymbol{\xi}_{\hat{\lambda}}^{(m,n)}) P_\mu(\boldsymbol{\xi}_{\hat{\lambda}}^{(m,n)}) \hat{\Delta}_{\hat{\lambda}}^{(m,n)}  = \\
&
\sum_{ m+n>\tilde{\lambda}_1>\tilde{\lambda}_2>\cdots >\tilde{\lambda}_n\geq 0 } \!\!\!\!\!\!\!\!
   \det \left[ \Psi^{(m+n-1)}_{\tilde{\lambda}_j,n-k+\lambda_k}\right]_{1\leq j,k\leq n}   \det \left[ \Psi^{(m+n-1)}_{\tilde{\lambda}_j,n-k+\mu_k}\right]_{1\leq j,k\leq n}   .
\end{align*}
 With the aid of the Cauchy-Binet formula, 
the latter sum is rewritten in terms of the determinant
\begin{align*}
&\det \left[   \sum_{\hat{l}\in\Lambda^{(m+n-1)} }   \Psi^{(m+n-1)}_{\hat{l},n-j+\lambda_j }   \Psi^{(m+n-1)}_{\hat{l},n-k+\mu_k}    \right]_{1\leq j,k\leq n} \\
&= \det \left[    \langle    \Psi^{(m+n-1)}_{\cdot ,n-j+\lambda_j } ,  \Psi^{(m+n-1)}_{\cdot ,n-k+\mu_k} \rangle_{(m+n-1)}   \right]_{1\leq j,k\leq n}
=\begin{cases}
0 & \text{if}\ \mu \neq \lambda ,\\
1 & \text{if}\ \mu = \lambda 
\end{cases} 
\end{align*}
($\lambda,\mu\in \Lambda^{(m,n)}$), where in the last step we relied  on the orthogonality in Remark \ref{dual:rem}.

\bibliographystyle{amsplain}

\end{document}